\documentclass[11pt]{amsart}

\usepackage{amsmath,amsfonts,amssymb,enumerate}
\usepackage{amsthm,dsfont}
\usepackage{latexsym, mathtools,verbatim,graphicx}
\usepackage{amscd,hyperref}
\usepackage{bbold}
\usepackage{bbm}
\usepackage{color}
\usepackage{enumitem}
\theoremstyle{plain}
\newtheorem{theorem}{Theorem}[section]
\newtheorem{thm}[theorem]{Theorem}

\newtheorem{corollary}[theorem]{Corollary}

\newtheorem{prop}[theorem]{Proposition}
\theoremstyle{definition}
\newtheorem{conj}[theorem]{Conjecture}

\theoremstyle{remark}

  {\hfill \fbox{}}

\newcommand{\C}{\mathbb{C}}

\newcommand{\T}{\mathbb{T}}
\newcommand{\D}{\mathbb{D}}
\newcommand{\N}{\mathbb{N}}

\newcommand{\Toep}{\mathrm{Toep}}

\begin{document}
\title[Toeplitz factors]{On the minimum number of Toeplitz factors of a matrix}

\author[Garc\'ia-Marco]{Ignacio Garc\'ia-Marco}
\author[M\'arquez-Corbella]{Irene M\'arquez-Corbella}
\author[Seco]{Daniel Seco}
\address{Universidad de la Laguna e IMAULL \newline  Avenida Astrof\'isico Francisco S\'anchez, s/n.  \newline Facultad de Ciencias, secci\'on: Matem\'aticas, apdo. 456.  \newline 38200 San Crist\'obal de La Laguna \newline
Santa Cruz de Tenerife,  Spain} \email{iggarcia@ull.edu.es} \email{imarquec@ull.edu.es} \email{dsecofor@ull.edu.es}

\begin{abstract}
We disprove a conjecture by Ye and Lim, by showing that there are $3 \times 3$ complex matrices which can't be expressed as the product of two Toeplitz matrices of the same size. We also improve previous estimates by Ye and Lim on the minimum number of Toeplitz matrices needed to factor any $n \times n$ matrix, for low values of $n$.
\end{abstract}

\thanks{The first and second authors are supported by the project MICINN PID2023-149508NB-I00, funded by MICIU/AEI/10.13039/501100011033 and by FEDER, EU. The third author is funded by grant PID2024-160185NB-I00 through the Generaci\'on de Conocimiento programme and by grant RYC2021-034744-I of the Ram\'on y Cajal programme from Agencia Estatal de Investigaci\'on (Spanish Ministry of Science, Innovation and Universities).}



\date{\today}

\maketitle

\section{Introduction}

We say that $M \in \C^{n \times n}$ is a \emph{Toeplitz} matrix if \begin{equation}\label{eqn100}
M_{i,j}=M_{i+1,j+1}\qquad \text{ for all } \qquad i,j=1,...,n-1.\end{equation} Such matrices play an important role in many areas of mathematics and their applications, and specific algorithms developed for this class of matrices are available \cite{HeinigRost}. In particular, some classical algorithms by Levinson and Schur provide very efficient solutions to linear systems given by a Toeplitz matrix (on the order of $n^2$ floating point operations). This would beat, for instance, usual general algorithms such as gaussian elimination, which require an amount of the order of $n^3$ flops. More modern approaches allow one to solve linear systems involving Toeplitz matrices in $\mathcal O(n {\rm log}^2(n))$ operations (see \cite{BO}). As a natural consequence, there is great interest on finding factorizations of other matrices as products of Toeplitz ones.
Decomposing matrices as products of structured matrices has been proven to be a very effective technique in computational and numerical algebra and, for example, decompositions of matrices as the product of a lower-triangular and an upper-triangular matrix (LU), or of an orthogonal and an upper-triangular matrix (QR), or of two orthogonal matrices with a diagonal one (SVD), are ubiquitous in many numerical applications. Toeplitz matrices are one of the most well-studied and understood classes of structured matrices and present several interesting computational advantages. The decompositions we will focus on have many attractive computational properties and are precious in computational algebra. We will call $(T_1, \dots , T_s)$ a \emph{Toeplitz decomposition} of $M$ (of length $s$) if $T_i$ is Toeplitz for $i=1,\dots, s$, and $M= T_1 \cdot T_2 \cdot \dots  T_s$. The interest lies in that if one knows such a decomposition for a matrix $M$, then one can solve linear systems in $M$ within $\mathcal O(sn ({\rm log}^2(n))$ operations. 

A notation we will use throughout the paper is that for a matrix $M \in \C^{n \times n}$, the \emph{Toeplitz number} of $M$, $\Toep(M)$, is the minimal number of Toeplitz matrices needed to factorize $M$ as their product, that is, the minimal length of its Toeplitz decompositions. When we take the maximum of $\Toep(M)$ over all $n \times n$ matrices, we obtain the \emph{maximal Toeplitz number} of order $n$, which we denote by $\Toep_n$. Until Ye and Lim's striking article \cite{YeLim}, one could have expected $\Toep(M)$ to be infinite for some matrix $M$, or that it may take unbounded values for a fixed $n$. However, they showed the following: 
\begin{thm}[Ye-Lim]\label{YeLim1} Let $n \in \N$ and $M \in \C^{n \times n}$. Then
\[\Toep_{n}\leq 4\left\lfloor \frac{n}{2} \right\rfloor+5.\]
 Moreover, if $M$ is invertible, then \[\Toep(M) \leq 2\left\lfloor \frac{n}{2} \right\rfloor+2;\] and finally, for a generic $M$, \[\Toep(M) \leq \left\lfloor \frac{n}{2} \right\rfloor +1.\]
\end{thm}
Here, genericity means that the exceptions are in an algebraic submanifold of dimension less than $n^2$. As the authors mention, no lower generic bound is possible. This follows from a simple dimension counting argument, and thus to some extent, the theorem means that Toeplitz matrices are well distributed for multiplication. The proof is non-constructive and it is our aim to focus on a more constructive approach. Of course, finding an explicit factorization is not as simple as proving its existence but this is a promising road for the future, nevertheless. 

It seems relevant to remark that the bound on $\Toep_n$ is clearly not exactly equal to its upper bound in all cases, since for $n=1$, every matrix is Toeplitz. The same applies for the bound for invertible matrices. For $n=2$, one can solve the problem of factorizing any matrix with two Toeplitz factors as an easy exercise. This means that the problem trully starts at $n=3$. Ye and Lim made the reasonable conjecture that their generic bound was true for any matrix, following the pattern observed for $n=1,2$ (see Conjecture 1 in \cite{YeLim}).

\begin{conj}[Ye-Lim]\label{conj1}
For all $n \in \N$, \[\Toep_n = \left\lfloor \frac{n}{2} \right\rfloor +1.\]
\end{conj}

In the present article we will show that this conjecture fails for $n=3$. In fact, it fails for the simplest non-Toeplitz matrix imaginable: the invertible matrix $M_3$ below, can't be decomposed as $M_3=T_1 T_2$ with $T_i$ Toeplitz:
\[M_3 := \begin{pmatrix}
1 & 0 & 0 \\ 0 & 2 & 0 \\ 0 & 0 & 3
\end{pmatrix}\]
This may discourage the reader to deal with this problem any further but the factorization of $M_3$ is not complicated at all if one allows for 3 factors. Indeed any diagonal matrix can be written as \begin{equation}\label{eqn300}\begin{pmatrix}
0 & 0 & d \\ 1 & 0 & 0 \\ 0 & 1 & 0 \end{pmatrix} \cdot \begin{pmatrix}
0 & 0 & e \\ 1 & 0 & 0 \\ 0 & 1 & 0 \end{pmatrix} \cdot \begin{pmatrix}
0 & 0 & f \\ 1 & 0 & 0 \\ 0 & 1 & 0 \end{pmatrix} = \begin{pmatrix}
d & 0 & 0 \\ 0 & e & 0 \\ 0 & 0 & f \end{pmatrix}.\end{equation}

Our results will include a description of several families of counterexamples to the conjecture for $n=3$ as well as improvements on the bounds for $n \leq 4$ based on rank observations (some of which are valid for larger values of $n$). We will also give a complete description of all the diagonal (and all the antisymmetric) matrices that contradict the conjecture for $n=3$, and show that for fixed $n \in \N$, the problem of finding the maximal Toeplitz number of order $n$ is solvable algorithmically by using Buchberger's algorithm. In our own opinion, a relevant finding here is that among diagonal matrices for $n=3$, counterexamples are generic, while among antisymmetric ones, the counterexamples are exceptional. It is natural to ask what is the solution for lower values of $n$ but the finiteness of the algorithm is only theoretical, since the problem of factorizing a $4 \times 4$ matrix as a product of 3 Toeplitz matrices of the same size seems already untractable. We will however describe much of the behavior of $3 \times 3$ matrices, and show that \begin{equation}\label{eqnToep3}3 \leq \Toep_3 \leq 4.\end{equation}
For matrices of $4\times 4$ size we will be able to show
\begin{equation}\label{eqnToep4}3 \leq \Toep_4 \leq 9.\end{equation}

We consider our approach to be more explicit than that of Ye and Lim, and we provide explicit factorizations of several families of matrices that exploit particular matrix structures.

Our plan is as follows: In Section \ref{opasandtoep} below, we provide insight on the operator theory motivation, behind our study of some particular matrices. Again $M_3$ is the simplest appearing matrix in the relevant class of matrices in that context. Then we proceed, in Section \ref{counter}, to show counterexamples to the conjecture by Ye and Lim. This will include a characterization of which diagonal  $3 \times 3$ matrices yield no factorization with 2 Toeplitz factors. We will comment on antisymmetric matrices, which can also be dealt with in reasonable time completely, to find which ones furnish counterexamples. We continue, in Section \ref{rankstuff}, with a few results lowering some of the upper bounds given by Ye and Lim for the number of needed Toeplitz factors in specific cases in terms of the rank. We make use of these observations to establish \eqref{eqnToep3} and \eqref{eqnToep4}. We conclude in Section \ref{further} with some remarks about questions that remain open.

\section{Toeplitz operators and optimal approximation in function spaces}\label{opasandtoep}

Part of our motivation to study the matrix $M_3$ above came from one of the applications of Toeplitz matrices in operator theory, where they represent the finite truncations of Toeplitz operators (projections composed with multiplication operators in spaces of analytic functions). Our initial observation is that Toeplitz matrices appear in the study of cyclic functions of certain space: Let $H$ be a Hilbert space of holomorphic functions over the unit disk $\D$. We say $f \in H$ is \emph{cyclic} if $\mathcal{P}f$ is dense in $H$. In many such spaces, determining whether a function $f$ is cyclic reduces to finding the assymptotic behavior of $1-p_n(0)f(0)$, where $p_n$ is the \emph{optimal polynomial approximant} (or \emph{opa}) to $1/f$ of degree $n \in \N$, that is, the unique polynomial of degree at most $n$ minimizing the norm \[\|1-pf\|_H\]
among all polynomials $p$ of degree at most $n$. 
It turns out (see \cite{BCLSS1, FMS1}) that the coefficients of $p_n$ are given by the solution $c$ to a linear system \[Mc=b,\] where $M$ is the matrix with entries \begin{equation}\label{eqn101} M_{j,k} = \left \langle z^k f, z^j f \right \rangle_H, \qquad j,k=0,\ldots,n,\end{equation}
and $b$ is the vector $\overline{f(0)}e_0$ ($e_j$ being the $j$-th element of the canonical basis in $\C^{n+1}$ for $j=0,\ldots, n$). Thus, algorithms for solving this linear system that exploit efficiently the structure of the matrices obtained for a particular space $H$ are highly sought after. This structure is linked to the behavior of the \emph{shift} operator $S$ taking a function $f$ to $Sf(z) = zf(z)$. Its relevance is due to its universality for certain classes of operators through spectral theorems.

A satisfactory solution to the above system has only been possible until now in the case of one particular space $H$: the Hardy space $H^2$ is the space formed by holomorphic functions in $\D$ which have a Maclaurin series with coefficients in $\ell^2$ (equipped with the corresponding $\ell^2$ norm). It can be interpreted as well in terms of the boundary values of the function belonging to $L^2(\T)$. It is one of the most studied objects in the complex analysis of the last century. This is in part thanks to a Theorem of Beurling that provides a complete understanding of the invariant subspaces for the shift operator acting there (and as a consequence, obtaining a complete characterization of cyclic functions). Invariant subspaces are, in this context, directly related with so-called inner functions, while cyclic functions are those called outer. Every $f \in H^2$ can be factorized as $f=IE$ where $I$ is inner, $E$ is outer and this provides all information about which invariant subspaces contain the function $f$. See \cite{Gar} for more on the basics of $H^2$. 

The space $H^2$ is uniquely special in that the shift operator is an isometry, and this means that the matrix $M$ with entries \eqref{eqn101} satisfies \eqref{eqn100}, i.e., $M$ is Toeplitz. An \emph{opa based} description of cyclicity in $H^2$ (coherent with Beurling's Theorem) was achieved in \cite{JLMS1}. From now on, we denote by $Z(g)$ the zero set of a polynomial function $g$. There it was shown that a function $f$ is cyclic if and only if 
\[\prod_{n=1}^\infty \prod_{z_k \in Z(p_n)} |z_k|^{-2} = \frac{|f(0)|^2}{\|f\|^2}.\]
In terms of the inner-outer factorization of $f=IE$, the right-hand side is maximized when the inner part is trivial, while the left hand-side does not depend on the inner part. The relevant feature here is that the zeros of opa (together with the quotient $|f(0)|/\|f\|$) determine the cyclic character of the function. The proof of this result was based on a recursive scheme that exploited the Toeplitz character of the matrices, by making a systematic use of Levinson's algorithm. 
Performing a similar recursion scheme in other spaces of analytic functions seems like a natural goal. We are going to focus now on the \emph{Dirichlet space} $D$: Let $f \in Hol(\D)$, and $f(z) = \sum_{k\in \N} a_k z^k$. Then $f$ is an element of $D$ if it satisfies \begin{equation}\label{eqn102}
\|f\|_D^2 := \sum_{k=0}^\infty |a_k|^2 (k+1) < \infty.
\end{equation} 
This induces a natural inner product making $D$ into a Hilbert space where it makes sense to study cyclic functions. In fact, a conjecture by Brown and Shields \cite{BS84} proposes a particular characterization of cyclic functions there, and several experts consider this a very relevant problem in the field, see \cite{EFKMR}. The opa approach was introduced with the long term goal of studying the cyclic functions in the Dirichlet space, and many tools from approximation theory could be used to study the conjecture if we had an available characterization of cyclic functions in terms of the zeros of opas, like the one for $H^2$.
When studying opas in $D$, there are two good reasons to expect some nice results: from the definition of the norm \eqref{eqn102} it is easy to see that the shift in $D$ is a 2-isometry, that is \[S^2f -2 Sf +f \equiv 0.\]
In terms of the matrices \eqref{eqn101} this means that 
\[M_{i,j}-M_{i+1,j+1}=M_{i+1,j+1}-M_{i+2,j+2}.\] In other words, rather than constant along each diagonal, these matrices vary at constant speed along each diagonal. We call such matrices, \emph{Dirichlet-Toeplitz} matrices (DTM). The simplest function one could test for its cyclicity is the constant 1 and this would provide a diagonal matrix $M_{n+1}$ with entries $1...n+1$.
The matrix $M_3$ obtained in the introduction is the result of taking $f=1$ and $n=2$. In that regard, it appears as one of the simplest DTM.  If one could describe any DTM $M \in \C^{n \times n}$ as a product of $k$  (independent of $n$) Toeplitz matrices, then one could hope describing cyclicity in $D$, by applying Levinson algorithm (as in $H^2$) $k$ times. For that matter, decomposing $M_n$ for each $n$ is particularly relevant. By simply counting dimensions, one could have hoped for the special class of DTM to be decomposed in as few as 2 Toeplitz factors. Being able to perform such a decomposition explicitly would have pointed towards a double application of Levinson as a solution to the problem of finding an opa based description of cyclicity in $D$. These and other questions were presented in the survey \cite{Seco2}. As per an idea in the introduction, any diagonal matrix can trivially be decomposed as the product of $n$ Toeplitz matrices. Even though we shall disprove such a simple decomposition with 2 factors, we still consider that decomposing $M_n$ explicitly as a product of Toeplitz matrices with $\Toep(M_n) (\leq n)$ factors is an interesting problem for larger values of $n$.  

\section{Counterexamples to the conjecture}\label{counter}

A simple and key observation is that given an $n \times n$ matrix $A$ and a value $s \in \mathbb Z^+$, determining whether $A$ admits a Toeplitz decomposition of length $s$ is equivalent to the problem of deciding if a particular system  of multivariate polynomial equations has complex solutions which, by Hilbert's Nullstellensatz (see, e.g., \cite[Chapter 4]{CLO}), amounts to determining whether $1$ belongs to  an ideal or not. More precisely:

\begin{prop}\label{pr:ecuacion} Let $A = (a_{i,j}) \in \mathcal M_{n \times n}(\C)$ be an $n \times n$ matrix and $s \geq 1$. Consider the set of $(2n-1)s$ variables $X = \{ x_{k,\ell} \, \vert, 1 \leq k \leq s,\, -(n-1) \leq \ell \leq n-1\}$, and the $n^2$ homogeneous polynomials of degree $s$ 
\[f_{i,j} :=  \sum_{1 \leq l_1, \ldots, l_{s-1} \leq n} x_{1, l_1-i}\, x_{2, l_2-l_1}\, \cdots\, x_{s,j-l_{s-1}} \in \C[X], \text{ for } i,j \in \{1,\ldots,n\}.  \]
The following conditions are equivalent:
\begin{itemize} 
\item The matrix $A \in \mathcal M_{n \times n}(\C)$ is a product of $s$ Toeplitz matrices.
\item The system of equations $f_{i,j} = a_{i,j}$ for all $i,j \in \{1,\ldots,n\}$ has a solution in $\C^{(2n-1)s}$.  
\item The ideal $I = \langle f_{i,j} - a_{i,j} \vert \, 1 \leq i,j \leq n \rangle \subseteq \C[X]$ satisfies that  $1 \notin I$.
\end{itemize}
\end{prop}

This result might suggest Gr\"obner basis techniques to compute Toeplitz ranks of matrices. However, due to the size of the problem, this approach seems to be useless even for quite small values of $n$ and $s$ ($n \geq 4$ and $s \geq 3$).  This did not deter us from trying to understand a few particular cases with this approach, but we restricted ourselves to small classes of small matrices.

In order to prove that certain matrix $M \in \C^{3 \times 3}$ is a counterexample to Conjecture \ref{conj1} when $n=3$ we need to show that the problem of finding Toeplitz matrices $T_1$ and $T_2$ such that $M=T_1T_2$ has no solution. Denote \[T_1 = \begin{pmatrix} a_2 & a_3 & a_4 \\ a_1 & a_2 & a_3 \\ a_0 & a_1 & a_2 \end{pmatrix}; T_2 = \begin{pmatrix} b_2 & b_3 & b_4 \\ b_1 & b_2 & b_3 \\ b_0 & b_1 & b_2  \end{pmatrix}; T_1T_2 -M= \begin{pmatrix} p_1 & p_2 & p_3 \\ p_4 & p_5 & p_6 \\ p_7 & p_8 & p_9 \end{pmatrix}.\] If we express the system $M=T_1T_2$ as a system of algebraic equations on the entries of $T_1$ and $T_2$ we are going to prove that there are no simultaneous solutions to all polynomials $p_1,\ldots, p_9$. In terms of Hilbert's Nullstellensatz, our problem is then equivalent to establishing the existence of polynomials $q_1,\ldots, q_9$ such that \begin{equation}\label{eqn200}1= \sum_{i=1}^{9} q_i p_i.\end{equation} Here $p_i$ will be each of the 9 entries of the product $T_1$, $T_2$, which will form a polynomial of degree 2 on the entries $a_0,\ldots, a_4; b_0,\ldots, b_4$. The polynomials $p_1,\ldots, p_9$ are given by
\[p_1 = a_2 b_2 + a_3 b_1 + a_4 b_0 -M_{1,1}, \qquad p_2 = a_2 b_3 + a_3 b_2 + a_4 b_1 -M_{1,2}\]
\[p_3 = a_2 b_4 + a_3 b_3 + a_4 b_2 -M_{1,3}, \qquad p_4 = a_1 b_2 + a_2 b_1 + a_3 b_0 -M_{2,1}\]
\[p_5 = a_1 b_3 + a_2 b_2 + a_3 b_1-M_{2,2}, \qquad p_6 = a_1 b_4 + a_2 b_3 + a_3 b_2 -M_{2,3}\]
\[p_7 = a_0 b_2 + a_1 b_1 + a_2 b_0 -M_{3,1}, \qquad p_8 = a_0 b_3 + a_1 b_2 + a_2 b_1 -M_{3,2}\]
\[p_9 = a_0 b_4 + a_1 b_3 + a_2 b_2 -M_{3,3}.\]
A proof that $M$ is a counterexample may thus just consist of a choice of polynomials $q_1,...,q_9$ such that  \eqref{eqn200} holds. By finding a Gr\"{o}bner basis for $I(p_1,\ldots,p_9)$ we will be able to determine whether there is a solution and therefore, whether $M$ is factorizable as the product of two Toeplitz matrices. 
The core of our message in this article is that for a given matrix, \emph{there is an algorithm} (Buchberger's) that determines whether or not a particular square $M$ may be written as a Toeplitz product $\prod_{j=1}^d  T_j$ for a particular $d \in \N$. One could leave a few elements of $M$ to be variables and the algorithm will still work. However, letting all 9 entries to be variables seems to generate already too large of a problem for our current computational power.  The algorithm we use is not fast and its complexity increases extremely fast on the dimension $n$ of $M \in \C^{n \times n}$. This is why we decided to focus on a few families of matrices and finding the counterexamples within such class.
\subsection{Diagonal matrices.} For diagonal matrices we can describe completely the value $\Toep(M)$. Remind that $\Toep(M)=1$ if and only if the matrix is Toeplitz, which for diagonal matrices only happens for multiples of the identity. Also, \eqref{eqn300} shows how to decompose any $3 \times 3$ diagonal matrix as the product of 3 factors (or, for that matter, any $n \times n$ diagonal matrix as the product of $n$ factors). Therefore, the question for these diagonal matrices is whether $\Toep(M)=2$ or $\Toep(M)=3$:
\begin{thm}\label{thmdiag}
Let $M= \begin{pmatrix}d & 0 & 0 \\ 0 & e & 0 \\ 0 & 0 & f\end{pmatrix}$. Then $\Toep(M)=3$ if and only if $d \neq e \neq f \neq d$ and $e \neq 0$.
\end{thm}

\begin{proof}

We make the choice of polynomials \[q_1 = (f-e)[a_2b_2(d-e) + a_3b_1(d-f)], \qquad q_2 = a_3b_0(d-f)(e-f), \qquad q_3 = a_2b_0(d-e)(e-f), \] \[q_4 = a_1b_4(d-f)(d-e) - a_2b_3(d-e)(e-f) + a_3b_2(d-f)(d+f-2e)\]\[q_5 = (d-f)[a_1b_3(e-d) + a_3b_1(e-f) +(f-e)(d-e)], \quad q_6 = (d-e)[a_2b_1(e -f) + a_3b_0(f-d)], \]
\[q_7 = a_2b_4(d-e)(f-e) + a_3b_3(d-f)(d+f-2e), \quad q_8 = (f-d)[a_1b_4(d-e) + a_3b_2(d+f-2e)],\] \[q_9 = (d-e)[a_1b_3(d-f)+ a_2b_2(e-f)].\]

In this way, we obtain 
\[\sum_{i=1}^9 q_i p_i = e(d-e)(f-d)(f-e).\]

In particular, if each of the factors on the right-hand side is different from 0, dividing by their product all $q_i$ values, we obtain that $Toep(M)=3$. In the remaining cases, $M$ must admit one of the following factorizations:
\[\begin{pmatrix}0 & d & 0 \\ 0 & 0 & d \\ f & 0 & 0\end{pmatrix} \cdot \begin{pmatrix}0 & 0 & 1 \\ 1 & 0 & 0 \\ 0 & 1 & 0\end{pmatrix} ; \quad  \begin{pmatrix}0 & 0 & d \\ 0 & 0 & 0 \\ f & 0 & 0\end{pmatrix} \cdot \begin{pmatrix}0 & 0 & 1 \\ 0 & 0 & 0 \\ 1 & 0 & 0\end{pmatrix}; \quad \begin{pmatrix}0 & 1 & 0 \\ 0 & 0 & 1 \\ 1 & 0 & 0\end{pmatrix} \cdot \begin{pmatrix}0 & 0 & f \\ d & 0 & 0 \\ 0 & d & 0\end{pmatrix};\] or, in the case $d=f$,  \[\begin{pmatrix}1 & 0 & 1 \\ 0 & 1 & 0 \\ \frac{e-d}{e} & 0 & 1 \end{pmatrix} \cdot \begin{pmatrix} e & 0 & -e \\ 0 & e & 0 \\ d-e & 0 & e\end{pmatrix}.\] 
\end{proof}

\subsection{Antisymmetric matrices.} We decided to include one further subclass of $\C^{3 \times 3}$ because it illustrates something specific: among diagonal matrices, Theorem \ref{thmdiag} tells us that counterexamples to Conjecture \ref{conj1} are generic. In other words, the exceptional set for the conjecture when $n=3$ has algebraic dimension at least $3$. 
In contrast, antisymmetric matrices, which do not intersect the (nonzero) diagonal matrices, display the opposite behavior—consistent with that of general matrices: generically, they satisfy the conjecture. 

However, within the $3$-dimensional space of antisymmetric $3\times 3$ matrices, there exists a $2$-dimensional subvariety consisting of counterexamples. Rather than presenting a full proof, we briefly mention that applying Buchberger’s algorithm to a general antisymmetric matrix of the form $$M= \begin{pmatrix} 0 & d & e \\ -d & 0 & f \\ -e & -f & 0\end{pmatrix}$$ 
provides a polynomial of degree 28 in the variables $d$, $e$ and $f$, such that the Ye-Lim conjecture holds outside its zero set. 

This polynomial admits a factorization into irreducible components of degrees $1$, $2$, and $4$. By analyzing all possible combinations of these factors, a complete classification of the counterexamples in this class can be obtained within reasonable computational effort.
Instead of presenting the full classification, we highlight two representative examples of counterexamples: any generic point on the surface defined by $df-e^2=0$ and any nonzero point on the line defined by $(d,e,f)=\lambda (-1,1,\sqrt{3}i)$.
We emphasize this case because computationally tractable instances of the conjecture are rare, and the antisymmetric case offers a particularly instructive and analyzable example.

\section{Rank and Toeplitz numbers}\label{rankstuff}

For low rank matrices one can devise explicit algorithms for finding Toeplitz  decompositions.

\begin{prop} 
\label{rank:1}
Let $M \in \mathcal M_{n \times n}(\C)$ of rank $1$, then $M = T_1 T_2 T_3,$ where $T_1, T_2 \hbox{ and } T_3$ are three Toeplitz matrices.
\end{prop}
\begin{proof}
Let \( M \in \mathcal{M}_{n \times n}(\mathbb{C}) \) be a matrix of rank \( 1 \). 
In this case, all rows of \( M \) are scalar multiples of any of its nonzero rows. Specifically, if $\mathbf{a}_i$ is a nonzero row of $M$:
\[
M = 
\begin{pmatrix}
\mathbf{a}_1 \\ \vdots \\ \mathbf{a}_n
\end{pmatrix}
= 
\begin{pmatrix}
\lambda_1 \mathbf{a}_i \\   \vdots \\ \lambda_i \mathbf{a}_i \\ \vdots \\ \lambda_n \mathbf{a}_1
\end{pmatrix}
= 
\underbrace{\begin{pmatrix}
\lambda_1 \\ \vdots \\ \lambda_i \\ \vdots \\ \lambda_n
\end{pmatrix}}_{\lambda}
\cdot 
\underbrace{\begin{pmatrix} a_{i1} & \cdots & a_{in} \end{pmatrix}}_{\mathbf{a}_i},
\]

with $\lambda_1 = 1$. This representation highlights that \( M \) can be factored into the product of a column vector and any of its nonzero rows, indicating its rank is \( 1 \).

We now construct three Toeplitz matrices:

\begin{enumerate}
    \item Using the column vector \( \lambda \), we define \( T_1 \) to be the only lower triangular Toeplitz matrix with $\lambda$ as its first column. In other words, \( T_1 \) has the form:
    \[
    T_1 = 
    \begin{pmatrix}
    \lambda_1 & 0 & 0 & \cdots & 0 \\
    \lambda_2 & \lambda_1 & 0 & \cdots & 0 \\
    \lambda_3 & \lambda_2 & \lambda_1 & \cdots & 0 \\
    \vdots & \vdots & \vdots & \ddots & 0 \\
    \lambda_n & \lambda_{n-1} & \lambda_{n-2} & \cdots & \lambda_1
    \end{pmatrix}.
    \]

 \item \( T_2 \) is a Toeplitz matrix with zeros in all positions except for the position \( (1, n) \), which contains a 1. The matrix \( T_2 \) has the form:
    \[
    T_2 =
    \begin{pmatrix}
    0 & 0 & 0 & \cdots & 1 \\
    0 & 0 & 0 & \cdots & 0 \\
    0 & 0 & 0 & \cdots & 0 \\
    \vdots & \vdots & \vdots & \ddots & \vdots \\
    0 & 0 & 0 & \cdots & 0
    \end{pmatrix}.
    \]
    \item Using the row vector \( \mathbf{a}_i \), we define \( T_3 \) to be the only lower triangular Toeplitz matrix with $\mathbf{a}_i$ as its last row. In this way, \( T_3 \) has the form:
    \[
    T_3 =
    \begin{pmatrix}
    a_{in} & 0 & 0 & \cdots & 0     \\
    a_{i,{n-1}} & a_{in} & 0 & \cdots & 0 \\
    \vdots & \vdots & \ddots & \vdots & \vdots \\
a_{i2} & a_{i3} & \cdots & a_{in} & 0 
\\    a_{i1} & a_{i2} & a_{i3} & \cdots & a_{in}
    \end{pmatrix}.
\]
\end{enumerate}

We have that \( M = T_1 \cdot T_2 \cdot T_3 \). 
\end{proof}

The following result directly follows from the above proposition.

\begin{corollary}
If $M \in \mathcal M_{n\times n}(\C)$ and $\mathrm{rank}(M) = 1$, then $\Toep(M) \leq 3$.
\end{corollary}

For $n \times n$ matrices of rank $n-1$, we have a different approach.

\begin{prop} 
\label{rank:n-1}
Let $M \in \mathcal M_{n \times n}(\C)$ with rank $n-1$. Then $M = P T_1 T_2,$ with $P$ an invertible matrix and $T_1,T_2$ two Toeplitz matrices.
\end{prop}

\begin{proof}
Let \( M \in \mathcal{M}_{n \times n}(\mathbb{C}) \) be a matrix of rank \( n-1 \).
Then, $M$ can be transformed into its row reduced echelon form (\(\mathrm{rref}(M)\)) via an invertible matrix $P$.
That is:
\[
M = P_0 \cdot \mathrm{rref}(M), \quad \text{with } P_0 \text{ an invertible matrix}.
\]
The matrix $\mathrm{rref}(M)$ has exactly $n-1$ pivots, and hence one of its rows is entirely zero. Let \( \mathcal{P} \subseteq \{1, \ldots, n\} \) be the set of 
indices where a pivot appears in \( \mathrm{rref}(M) \); then $\mathcal P$ has cardinality \( n-1 \), so exactly one column index is missing. We analyse two cases: 

\begin{enumerate}
\item \textbf{Case 1:} $\mathcal{P} = \{2, \ldots, n\}.$
In this case, \( \mathrm{rref}(M) \) is already a Toeplitz matrix and the result holds.

\item \textbf{Case 2:} $\mathcal P$ is the union of two blocks of consecutive indices, separated by one missing index (the second block, potentially empty if $n=r+1$):
\[
\mathcal{P} = \{1, \ldots, r\} \cup \{r+2, \ldots, n\}, \quad \text{with } 1 \leq r < n \text{ and } s = n-1-r \geq 0.
\]
\end{enumerate}

We are interested in \textbf{Case 2}, where we have that 

\[
\mathrm{rref}(M) =
\left( \begin{matrix}
I_r & \mathbf{a} & 0_{r \times s} \\
0_{s \times r} & 0_{s \times 1} & I_s \\
0_{1 \times r} & 0 & 0_{1 \times s}
\end{matrix}\right) \hbox{ with } \mathbf a = \left( \begin{matrix} a_{1,r+1} \\ \vdots \\ a_{r,r+1}\end{matrix}\right),
\]
where \( I_m \) is the identity matrix of size \( m \times m \), and \( 0_{p \times m} \) denotes a block of zeros of size \( p \times m \).
What remains is to check that $\mathrm{rref}(M)$ \emph{is row equivalent to another one} that can be decomposed as the product of two Toeplitz matrices. Therefore, we are allowed to perform row operations on the $\mathrm{rref}(M)$ and this will only affect the choice of invertible matrix $P_0$, substituting it by another invertible matrix $P$.

Thus we now construct an auxiliary matrix $A$, row-equivalent to $\mathrm{rref}(M)$ by performing the following row operations:

\begin{enumerate}
    \item \textbf{Initial Row Reordering:} Move \( \mathrm{Row}_n \) to position \( r+1 \), shifting \( \mathrm{Row}_i \) with \( i \in\{ r+1, \ldots, n-1\} \) one position down each.
    \item \textbf{First Row Operation:} 
    \[
    \mathrm{Row}_i \to \mathrm{Row}_i + a_{r,r+1} \cdot \mathrm{Row}_{i+1}, \quad \text{for } i = 1, \ldots, r-1, r+2, \ldots, n-1.
    \]
\end{enumerate}
From now on, denote by $S_m$ the \emph{shift} matrix of size $m \times m$, that is, the matrix with $1$ on the diagonal above the main one, and zero elsewhere. We also denote by $B_m$ its transpose (the \emph{backward shift}). Notice that $S_m^k$ (respectively, $B_m^k$) is a Toeplitz matrix with entries equal to 1 on the $k$-th diagonal above (resp., below) the main one and 0 elsewhere.
In this way, $\mathrm{rref}(M)$ is transformed into:
\[  A^{(2)} =
\left( \begin{matrix}
I_r + a_{r,r+1} S_r & \mathbf{a^{(2)}} & 0_{r \times s} \\[8pt]
0_{1 \times r} & 0 & 0_{1 \times s} \\[8pt]
0_{s \times r} & 0_{s \times 1} & I_s + a_{r,r+1} S_s
\end{matrix}\right),
\]
where the vector \( \mathbf{a^{(2)}} \) is updated as:
    \[
    \mathbf{a^{(2)}} =
   (I_r+a_{r,r+1} B_r) \cdot \mathbf a.\]

We continue with additional row operations.
\begin{enumerate}[label=(\alph*),start=3]
\item The \textbf{General \( k \)-th Row Operation} is as follows:
    \[
    \mathrm{Row}_i \to \mathrm{Row}_i + \mathbf{a}_{r-k+1}^{(k+2)} \cdot \mathrm{Row}_{i+(k+2)},
    \hbox{ for } 
 i = 1, \dots, r-k , r+k+2, \dots, n-(k+2),\] where the vector \( \mathbf{a}^{(t)} \) is updated as:
    \[
    \mathbf{a}^{(t)} =
     \mathbf a^{(t-1)} + \mathbf{a}^{(t-1)}_{r+1-t} B_r^t \mathbf a.\]
\end{enumerate}
      
After all the operations, the resulting matrix is:
\[A^{(r)} =
\left( \begin{matrix}
I_r + U^{(r)}_r & \mathbf{a^{(r)}} & 0_{r \times s} \\[8pt]
0_{1 \times r} & 0 & 0_{1 \times s} \\[8pt]
0_{s \times r} & 0_{s \times 1} & I_s + U^{(r)}_s
\end{matrix}\right),
\] 
where: (1) \( \mathbf{a^{(r)}} \) is a column vector of size \( r \times 1 \), iteratively updated by the row operations and
(2) \( U^{(r)}_m \) is an upper triangular matrix of size \( m \times m \) with: \( a_{r,r+1} \) on the first superdiagonal, \( a_{r-1}^{(2)} \) on the second superdiagonal, \( a_{r-2}^{(3)} \) on the third one, and so on until  \( a_{1}^{(r)} \) appears on the \( (r-1) \)-th  one.

Observe that the matrix \(A^{(r)} \) is not Toeplitz, but the property \eqref{eqn100} is only broken at the elements of \(\mathrm{Row}_{r+1}\) and on the top right block. Let's first deal with the latter. Because the lower right block is invertible, $A^{(r)}$ is row equivalent to another matrix $ \hat{A}^{(r)}$ of the form

 \[\hat{A}^{(r)} =
\left( \begin{matrix}
I_r + U^{(r)}_r & \mathbf{a^{(r)}} & B \\[8pt]
0_{1 \times r} & 0 & 0_{1 \times s} \\[8pt]
0_{s \times r} & 0_{s \times 1} & I_s + U^{(r)}_s
\end{matrix}\right),
\] 
where $B$ is chosen so that the matrix formed by $\mathbf{a}^{(r)}$ and $B$ has the Toeplitz property \eqref{eqn100}.

\begin{enumerate}[label=(\alph*),start=4]
\item \textbf{Final Row Reordering:} Move the last $s$ rows of $A^{(r)}$ to the top and the first $r$ rows of $A^{(r)}$ to the botton.
\end{enumerate}

Thus the resulting matrix is:
\[ 
A =
\left( \begin{matrix}
0_{s \times r} & 0_{s \times 1} & I_s + U^{(r)}_s
 \\[8pt]
0_{1 \times r} & 0 & 0_{1 \times s} \\[8pt]
I_r + U^{(r)}_r & \mathbf{a^{(r)}} & B
\end{matrix}\right).
\]

The matrix $A$ is row equivalent to $\mathrm{rref}(M)$ and thus there is an invertible matrix $P$ making $M=PA$. What remains now is to decompose $A$ as the product of two Toeplitz matrices.

The first of these will be the Toeplitz matrix $T_1$ given by 
\[
T_1 =
\left( \begin{matrix}
0_{s \times r} & 0_{s \times 1} & I_s \\
0_{1 \times r} & 0 & 0_{1 \times s} \\
I_r & 0_{r \times 1} & 0_{r \times s}
\end{matrix}\right),
\]
Note that, if one multiplies $T_1$ by a given matrix $C$ of size $n\times n$, the multiplication $T_1 \cdot C$ rearranges the rows of $C$ as follows: (1) The last $s$ rows of $C$ are moved to the top; (2) the $s+1$-th row becomes a completely null row; and (3) the first $r$ rows of $C$ are moved to the bottom.

Now, our second Toeplitz matrix, \( T_2 \), is simply the result of adding a vector on row $r+1$ of $\hat{A}^{(r)}$ to transform it into a Toeplitz matrix. 
Specifically, we replace \(\mathrm{Row}_{r+1}\) with an adequate vector $(0_{1 \times r},  1, (\mathbf{a}^{(r)})')$ such that each descending diagonal of \( T_2 \) contains the same value, ensuring the Toeplitz property is satisfied.

In this way, it can be verified that $$M=P \cdot T_1 \cdot T_2.$$
\end{proof}

Now, using Theorem \ref{YeLim1} and the above Proposition we can deduce the following result.

\begin{corollary}
If $M \in \mathcal M_{n\times n}(\C)$ and $\mathrm{rank}(M) = n-1$, then $\Toep(M) \leq 2\lfloor \frac{n}{2} \rfloor + 4$.
\end{corollary}

\subsection{Toeplitz decomposition of arbitrary matrices of size $3\times 3$}

Let $M$ be a $3\times 3$ matrix. If $M$ has rank $1$, then by Proposition \ref{rank:1}, $M$ can be expressed as the product of three Toeplitz matrices. If $M$ has rank $3$, then $M$ is invertible and, thus, $\Toep(M) \leq 2\lfloor \frac{3}{2} \rfloor + 2 = 4$ (by Theorem \ref{YeLim1}). Therefore, apart from possibly improving on the estimate for invertible matrices, the remaining question is what is the behavior of rank $2$ matrices. Once we prove the following we will know that $Toep_3\leq 4$.

\begin{prop}
Let $M\in \mathcal M_{3\times 3}(\C)$ with $\mathrm{rank}(M) = 2$, then $\Toep(M) \leq 4$.
\end{prop}

\begin{proof}
Let $M\in \mathcal M_{3\times 3}(\C)$ with $\mathrm{rank}(M) = 2$. In this case, two columns are linearly independent, while the third is a linear combination of the other two.  
First, we assume that columns 1 and 3 are the independent ones and $M$ has the following form:

\[
M = \begin{pmatrix}
a & b & c \\ 
d & e & f \\ 
g & h & i
\end{pmatrix} = \begin{pmatrix} 
C_1 & | & \lambda \cdot C_1 + \beta\cdot C_3 & | & C_3 
\end{pmatrix}.
\] where $C_i$ represents column $i$. We shall employ a variety of techniques. The following two strategies consist on identifying structure in $M$ that can be factored with one Toeplitz matrix and an \emph{easier} version of $M$. 
\begin{enumerate}
\item \underline{Strategy One}\[
\begin{pmatrix} C_1 & | & 0 & | & C_3 \end{pmatrix} 
\begin{pmatrix} 
1 & \lambda & 0 \\ 
\beta & 1 & \lambda \\ 
0 & \beta & 1 
\end{pmatrix} 
= 
\begin{pmatrix} 
C_1 & | & \lambda \cdot C_1 + \beta\cdot C_3 & | & C_3 
\end{pmatrix}
\]

\item \underline{Strategy Two}
\[\begin{pmatrix} C_1 & | & 0 & | & C_3 \end{pmatrix} 
\begin{pmatrix} 
0 & \lambda & 1 \\ 
\beta & 0 & \lambda \\ 
1 & \beta & 0 
\end{pmatrix} 
= 
\begin{pmatrix} 
C_3 & | & \lambda \cdot C_1 + \beta\cdot C_3 & | & C_1 
\end{pmatrix}
\]
\end{enumerate}

We will analyze different cases:  
\begin{enumerate}
    \item \textbf{Case \(c \neq 0 \neq g\):} 
Note that:

\[
\underbrace{\begin{pmatrix}
g & d & a \\ 
\frac{f\cdot g}{c} & g & d \\ 
\frac{i\cdot g}{c} & \frac{f\cdot g}{c} & g
\end{pmatrix}}_{T_1} \cdot 
\underbrace{\begin{pmatrix}
0 & 0 & \frac{c}{g} \\ 
0 & 0 & 0 \\ 
1 & 0 & 0
\end{pmatrix}}_{T_2} =
\begin{pmatrix}
a & 0 & c \\ 
d & 0 & f \\ 
g & 0 & i
\end{pmatrix}.
\]

Now, by using Strategy One, we can express $M$ as the product of three Toeplitz matrices.

    \item \textbf{Case \(c = g= 0\):} 

\[
\underbrace{\begin{pmatrix}
0& d & a \\ 
f & 0 & d \\ 
i & f & 0
\end{pmatrix}}_{T_1} \cdot 
\underbrace{\begin{pmatrix}
0 & 0 & 1\\
0 & 0 & 0 \\
1 & 0 & 0
\end{pmatrix}}_{T_2} =
\begin{pmatrix}
a & 0 & 0 \\ 
d & 0 & f \\ 
0 & 0 & i
\end{pmatrix}.
\]

Again with Strategy One we achieve our goal.

For the remaining cases, we can now suppose $c=0 \neq g$. Otherwise, we just need to transpose each Toeplitz factor and invert the order of the factorization.
    \item \textbf{Case \(c = 0 \neq agi\):} 

\[
\underbrace{\begin{pmatrix}
a & \frac{a\cdot f}{i} & 0 \\ 
d & a & \frac{a\cdot f}{i} \\ 
g & d & a
\end{pmatrix}}_{T_1} \cdot 
\underbrace{\begin{pmatrix}
0 & 0 & 1\\
0 & 0 & 0 \\
\frac{i}{a} & 0 & 0
\end{pmatrix}}_{T_2} =
\begin{pmatrix}
0 & 0 & a \\ 
f & 0 & d \\ 
i & 0 & g
\end{pmatrix}.
\]

In this case we need to apply Strategy Two to achieve our goal.

    \item \textbf{Case \(a=c =i= 0 \neq g \):} Note that if \(c = a = 0\), then it follows that \(b = 0\) since the second column is a linear combination of the other two.

\[
\underbrace{\begin{pmatrix}
0 & f & 0 \\ 
d & 0 & f \\ 
g & d & 0
\end{pmatrix}}_{T_1} \cdot 
\underbrace{\begin{pmatrix}
0 & 0 & 1\\
0 & 0 & 0 \\
1 & 0 & 0
\end{pmatrix}}_{T_2} =
\begin{pmatrix}
0 & 0 & 0 \\ 
f & 0 & d \\ 
0 & 0 & g
\end{pmatrix}.
\]

Now we apply Strategy Two and we are done.

    \item \textbf{Case \( a = c =0 \neq dgi\):} again in this case $b=0$ and we have

\[
\begin{pmatrix}
0 & 0 & 0\\
d & e & f\\
g & h & i
\end{pmatrix} = 
\underbrace{\begin{pmatrix}
0 & 1 & 0 \\
0 & 0 & 1 \\
1 & 0 & 0
\end{pmatrix}}_{T_1} \cdot 
\underbrace{\begin{pmatrix}
0 & 0 & \frac{i}{d} \\
0 & 0 & 0 \\
1 & 0 & 0
\end{pmatrix}}_{T_2} \cdot
\underbrace{\begin{pmatrix}
d & e & f \\
\frac{hd}{i} & d & e \\
\frac{gd}{i} & \frac{hd}{i} & d
\end{pmatrix}}_{T_3}.
\]
If $i=0$ but $g\neq 0$ then reverse the order and the roles of the symbols.

\item \textbf{Case \(a =c =  d = 0 \neq gi\):} again in this case $b=0$ and we can assume that $f\neq 0$ for the same reason. Now we have that:

\[
\begin{pmatrix}
0 & 0 & 0\\
0 & e & f\\
g & h & 0
\end{pmatrix} = 
\underbrace{\begin{pmatrix}
0 & 0 & 0 \\
1 & 0 & 0 \\
0 & 1 & 0
\end{pmatrix}}_{T_1} \cdot 
\underbrace{\begin{pmatrix}
0 & 0 & \frac{f}{g} \\
1 & 0 & 0 \\
0 & 1 & 0
\end{pmatrix}}_{T_2} \cdot
\underbrace{\begin{pmatrix}
g & h & i \\
\frac{eg}{f} & g & h \\
0 & \frac{eg}{f} & g
\end{pmatrix}}_{T_3}.
\]

\item \textbf{Case \(c = i = 0 \neq ag\):} in this case $f\neq 0$. Now we have that:

\[
\begin{pmatrix}
a & b & 0\\
d & e & f\\
g & h & 0
\end{pmatrix} = 
\underbrace{\begin{pmatrix}
0 & 0 & 1 \\
1 & 0 & 0 \\
0 & 1 & 0
\end{pmatrix}}_{T_1} \cdot 
\underbrace{\begin{pmatrix}
d & e & f \\
g & h & 0 \\
a & b & 0
\end{pmatrix}}_{B}. \]
Since $B$ is under the hypotheses of case (a), then $M$ is a product of $4$ Toeplitz matrices.
\end{enumerate}

Now, we assume that $M$ has rank $2$ and columns 1 and 3 are proportional.
One can write 

\[
\begin{pmatrix}
a & b & c\\
d & e & f\\
g & h & i
\end{pmatrix} = 
\underbrace{\begin{pmatrix}
 b & c & a \\
 e & f & d \\
 h & i & g
\end{pmatrix}}_{B} \cdot 
\underbrace{\begin{pmatrix}
0 & 1 & 0 \\
0 & 0 & 1 \\
1 & 0 & 0
\end{pmatrix}}_{T_1}. \]
The first and third columns of $B$ are not proportional and $B$ has rank $2$. Hence, if $B$ does not fall in the case (g), then it can be written as a product of 3 Toeplitz matrices and, therefore, $A$ is a product of $4$ Toeplitz matrices.

If $B$ falls in case (g), then $M$ has the following shape:
\[ M = 
\begin{pmatrix}
0 & b & 0\\
d & e & f\\
0 & h & 0
\end{pmatrix},
\]
with $b,h \neq 0$.

In the latter case, we write 
\[
M = 
\underbrace{\begin{pmatrix}
 0 & b & b \\
 d & d+e & f+e\\
 0 & h & h
\end{pmatrix}}_{B} \cdot 
\underbrace{\begin{pmatrix}
1 & -1 & 1 \\
0 & 1 & -1 \\
0 & 0 & 1
\end{pmatrix}}_{T_1} \]
and $B$ falls in case (a); thus $M$ is a product of four Toeplitz factors.
\end{proof}

All the preceding analysis supports the validity of the following theorem:
\begin{thm} \label{thm:toep3}
$3 \leq \Toep_3 \leq  4$.
\end{thm}

\subsection{Toeplitz decomposition of matrices of size $4\times 4$}

Let $M$ be a $4\times 4$ matrix. If $M$ has rank $1$, then by Proposition \ref{rank:1}, $M$ can be expressed as the product of three Toeplitz matrices. If $M$ has rank $3$, then by Proposition \ref{rank:n-1}, $\Toep(M) \leq 8$ . If $M$ has rank $4$, then $M$ is invertible and, thus $\Toep(M) \leq 6$ (by Theorem \ref{YeLim1}). Therefore, we set ourselves the task of studying the remaining case: to study Toeplitz decompositions of matrices of rank $2$. 

\begin{prop} Let $M\in \mathcal M_{4\times 4}(\C)$ with $\mathrm{rank}(M) = 2$, then $\Toep (M ) \leq 9$.
\end{prop}

\begin{proof}
If $M$ has rank $2$, several cases arise, depending on the row reduced echelon form of $M$:

\begin{itemize}

\item Case 1: $\mathrm{rref}(M) = \left( \begin{array}{cccc} 
1 & 0 & a & b\\
0 & 1 & c & d\\
0 & 0 & 0 & 0\\
0 & 0 & 0 & 0\\
\end{array}\right)$. This case divides into 3 possibilities.

\begin{enumerate}
\item If $d\neq 0$, then there exists an invertible matrix $P$ such that:
$$M= P \cdot \left(\begin{array}{cccc} 
0 & d' & c' & 1 \\
0 & 0 & 0 & 0\\
0 & 0 & 0 & 0\\
1 & 0 & a & b \end{array}\right) =  P \cdot 
\left( \begin{array}{cccc} 
0 & 0 & 0 & 1 \\ 0 & 0 & 0 & 0\\ 0 & 0 &0 & 0\\ 1 & 0 & 0 & 0
\end{array}\right) \cdot 
\left( \begin{array}{cccc} 
1 & 0 & a & b \\ c' & 1 & 0 & a\\ d' & c' & 1 & 0\\ 0 & d' & c' & 1
\end{array}\right) \cdot 
$$
where $d'=\frac{1}{d}$ and $c' = \frac{c}{d}$.
\item If $d=0 \neq c$, then there exists an invertible matrix $P$ such that:
$$M= P \cdot  \underbrace{\left(\begin{array}{cccc} 
0 & 0 & c' & 1\\
0 & 0 & 0 & 0\\
0 & 0 & 0 & 0\\
1 & 0 & a & b\\
\end{array}\right)}
=
P \cdot \underbrace{\left( \begin{array}{cccc}
0 & 0 & 1 & 0 \\
0 & 0 & 0 & 0\\
0 & 0 & 0 & 0\\
1 & 0 & 0 & 0
\end{array}\right)}_{M_1}
\cdot \left(\begin{array}{cccc} 
1 & 0 & a & b \\
c' & 1 & 0 & a\\
0 & c' & 1 & 0 \\
0 & 0 & c' & 1
\end{array}\right)
$$
where $c'= \frac{1}{c}$. In this case $M_1$ is not a Toeplitz Matrix but
$$M_1 = 
\left(\begin{array}{cccc} 
0 & 0 & 0 & 1\\
0 & 0 & 0 & 0 \\
0 &0 &0 &0 \\
1 & 0 & 0 & 0
 \end{array} \right) \cdot 
\left(\begin{array}{cccc}  
0 & 0 & 1 & 0\\
0 & 0 & 0 & 1\\
0 & 0 & 0 & 0\\
1 & 0 & 0 & 0 
\end{array}\right). $$
\item If $c=d=0$, then adding to the second row $b$ times the first row, we know that there exist invertible matrices $P, P'$ such that
$$M= P \cdot \mathrm{rref}(M) = P'\cdot \left( 
\begin{array}{cccc}
1 & 0 & a & b \\
0 & 0 & 0 & 0 \\
0 & 0 & 0 & 0 \\
b & 1 & ab & b^2
\end{array}
\right) = P' \left( \begin{array}{cccc}0 & 0 & 0 & 1\\ 
0 & 0 & 0 & 0 \\
0 & 0 & 0 & 0 \\
1 & 0 & 0 & 0\end{array}\right) 
\left(\begin{array}{cccc} 
b & 1 & ab &  b^2 \\
a & b & 1 & ab \\
0 & a & b & 1 \\
1 & 0  & a & b
\end{array}\right).$$
\end{enumerate}


\item Case 2: $\mathrm{rref}(M) = \left( \begin{array}{cccc} 
1 & a & 0 & b\\
0 & 0 & 1 & c\\
0 & 0 & 0 & 0\\
0 & 0 & 0 & 0\\
\end{array}\right)$. This case divides into 2 further possibilities:

\begin{enumerate}
    \item If $c\neq 0$, then there exists an invertible matrix $P$ such that:
$$M = P \cdot \left(\begin{array}{cccc} 
0 & 0 & c' & 1 \\
0 & 0 & 0 & 0 \\
0 & 0 & 0 & 0 \\
1 & a & 0 & b \\
\end{array}\right) = P \cdot 
 \left( \begin{array}{cccc}0 & 0 & 0 & 1\\ 
0 & 0 & 0 & 0 \\
0 & 0 & 0 & 0 \\
1 & 0 & 0 & 0\end{array}\right)
\left( \begin{array}{cccc} 
1 & a & 0 & b \\
c' & 1 & a & 0 \\
0 & c' & 1 & a \\
0 & 0 & c' & 1
\end{array}\right)
$$
where $c' = \frac{1}{c}$.
\item If $c=0$, then adding to the second row $b$ times first row, we know that there exist invertible matrices $P, P'$ such that  
$$M= P \cdot \mathrm{rref}(M) = P'\cdot \left( 
\begin{array}{cccc}
1 & a & 0 & b \\
0 & 0 & 0 & 0 \\
0 & 0 & 0 & 0 \\
b & ab & 1 & b^2\\
\end{array}
\right) = P' \left( \begin{array}{cccc}0 & 0 & 0 & 1\\ 
0 & 0 & 0 & 0 \\
0 & 0 & 0 & 0 \\
1 & 0 & 0 & 0\end{array}\right) 
\left(\begin{array}{cccc} 
b & ab & 1 & b^2 \\
0 & b & 1 & ab \\
a & 0 & b & 1 \\
1 & a & 0 & b
\end{array}\right).$$
\end{enumerate}


\item Case 3: $\mathrm{rref}(M) = \left( \begin{array}{cccc} 
1 & a & b & 0\\
0 & 0 & 0 & 1\\
0 & 0 & 0 & 0\\
0 & 0 & 0 & 0\\
\end{array}\right)$. In this case, there exists an invertible matrix $P$ such that 

$$M= P \cdot \left(\begin{array}{cccc} 

1 & a & b & 0 \\
0 & 0 & 0 & 0 \\
0 & 0 & 0 & 0\\
0 & 0 & 0 & 1 \\
\end{array} \right)  = P \cdot 
\left( \begin{array}{cccc} 
1 & 0 & 0 & 0 \\ 0 & 0 & 0 & 0\\ 0 & 0 &0 & 0\\ 0 & 0 & 0 & 1
\end{array}\right) \cdot 
\left(\begin{array}{cccc} 
0 & 0 & 0 & 1\\
b & 0 & 0 & 0 \\
a & b & 0 & 0 \\
1 & a & b & 0
\end{array} \right).$$


\item Case 4: $\mathrm{rref}(M) = \left( \begin{array}{cccc} 
0 & 1 & 0 & a\\
0 & 0 & 1 & b\\
0 & 0 & 0 & 0\\
0 & 0 & 0 & 0\\
\end{array}\right)$. Here we have an invertible matrix $P$ such that 

$$M= P \cdot \left(\begin{array}{cccc} 

0 & 0 & 0 & 0 \\
0 & 0 & 0 & 0 \\
0 & 1 & b & a'\\
0 & 0 & 1 & b \\
\end{array} \right)  = P \cdot 
\left( \begin{array}{cccc} 
0 & 0 & 0 & 0 \\ 0 & 0 & 0 & 0\\ 1 & 0 &0 & 0\\ 0 & 1 & 0 & 0
\end{array}\right) \cdot 
\left(\begin{array}{cccc} 
0 & 1 & b & a'\\
0 & 0 & 1 & b \\
0 & 0 & 0 & 1 \\
0 & 0 & 0 & 0
\end{array} \right).$$


\item Case 5: $\mathrm{rref}(M) = \left( \begin{array}{cccc} 
0 & 1 & a & 0\\
0 & 0 & 0 & 1\\
0 & 0 & 0 & 0\\
0 & 0 & 0 & 0\\
\end{array}\right)$, then there exists an invertible matrix $P$ such that 

$$M= P \cdot \left(\begin{array}{cccc} 
0 & 1 & a & 0\\
0 & 0 & 0 & 0 \\
0 & 0 & 0 & 0 \\
0 & 0 & 0 & 1
\end{array} \right)  = P \cdot 
\left( \begin{array}{cccc} 
0 & 0 & 0 & 1 \\ 0 & 0 & 0 & 0\\ 0 & 0 &0 & 0\\ 1 & 0 & 0 & 0
\end{array}\right) \cdot 
\left(\begin{array}{cccc} 
0 & 0 & 0 & 1\\
a & 0 & 0 & 0\\
1 & a & 0 & 0 \\
0 & 1 & a & 0
\end{array} \right).$$

\item Case 6: $\mathrm{rref}(M) = \left( \begin{array}{cccc} 
0 & 0 & 1 & 0\\
0 & 0 & 0 & 1\\
0 & 0 & 0 & 0\\
0 & 0 & 0 & 0\\
\end{array}\right)$. In this case, $\mathrm{rref}(M)$ is itself a Toeplitz matrix, and there exists an invertible matrix $P$ such that 
$M=P\cdot \mathrm{rref}(M)$.

\end{itemize}
\end{proof}

All the information we have about $4 \times 4$ matrices crystallizes into the following Theorem:

\begin{thm}\label{thm:toep4}
$3 \leq \Toep_4 \leq 9$.
\end{thm}


\section{Some further questions}\label{further}

\subsection{Chessboard-type constructions for diagonal matrices.} Consider the analogue of $M_3$ in $\C^{n \times n}$, $M_n$, a diagonal matrix with entries $1,\ldots, n$. From \eqref{eqn300} one could be lead to believe that $M_n$ is always a difficult matrix to factorize, since the following factorization in $n$ factors is very natural, and thus finding a shorter expression may seem unlikely: \[M_n= \prod_{j=1}^n Q_j, \qquad Q_j = \begin{pmatrix} 0 & j \\ I & 0\end{pmatrix}\]
In this regard, we can show that $M_4$ is not a counterexample to Conjecture \ref{conj1} for $n=4$.
Indeed, $M_4 = T_1 T_2 T_3$ where
\[\begin{array}{ccc}
T_1 = \begin{pmatrix} 0 & -1 & 0 & 2 \\ 1 & 0 &  -1 & 0 \\ 0 & 1 & 0 & -1 \\ \frac{-1}{3} & 0 & 1 & 0 \end{pmatrix},& T_2 = \begin{pmatrix} 0 & 1 & 0 & 3 \\ 1 & 0 &  1 & 0 \\ 0 & 1 & 0 & 1 \\ \frac{1}{2} & 0 & 1 & 0 \end{pmatrix},& T_3 = \begin{pmatrix} 0 & 0 & 6 & 0 \\ 0 & 0 &  0 & 6 \\ 1 & 0 & 0 & 0 \\ 0 & 1 & 0 & 0 \end{pmatrix}.\end{array}\]

In this example, we notice that the rows and columns of $T_1$, $T_2$ and $T_3$ divide into two mutually orthogonal families of vectors: even columns of $T_1$ and $T_2$ and odd ones of $T_3$ are orthogonal to the rest. We call this pattern, a \emph{chessboard}, and such constructions seem likely to appear when factorizing any chessboard-type matrix (in particular, diagonal ones), since the dimension counts are improved in favour of a possible factorization. We were able to show that \emph{generic} $4 \times 4$ diagonal matrices (with diagonal $d_1, d_2, d_3, d_4$) can be factored with 3 explicit chessboard Toeplitz matrices.
The following prompt in Sage:

 {\small \begin{verbatim}
   R.<a0,a6,b0,b6,c1,c5,d1,d2,d3,d4>=PolynomialRing(QQ,10,order='lex');
   A=matrix(R,[[0,1,0,a6],[1,0,1,0],[0,1,0,1],[a0,0,1,0]]);
   B=matrix(R,[[0,1,0,b6],[-1,0,1,0],[0,-1,0,1],[b0,0,-1,0]]);
   C=matrix(R,[[0,0,c5,0],[0,0,0,c5],[c1,0,0,0],[0,c1,0,0]]);
   M=matrix(R,[[d1,0,0,0],[0,d2,0,0],[0,0,d3,0],[0,0,0,d4]]);
   S=A*B*C-M;
   L = [S[i][j] for i in range(4) for j in range(4)];
   I=R.ideal(L)
   G=I.groebner_basis('libsingular:slimgb')
   G
\end{verbatim}}

\noindent outputs a Gröbner basis $G$ with $20$ polynomials satisfying that $G \cap \mathbb C[d_1,d_2,d_3,d_4] = \emptyset$ (which means that $J \cap \mathbb C[d_1,d_2,d_3,d_4] = \{0\}$, see, e.g., \cite[Theorem 2 in Chapter 3\S 1]{CLO}), and including the following polynomials:
\begin{itemize}
\item $f_1 =  c_5 (d_1 d_4 - d_2 d_3) + d_1 d_3 d_4 - d_2 d_3 d_4$,
\item $f_2 = c_1 d_4 - c_5 d_2 - d_2 d_4$,
\item $f_3 = b_6 (d_1 d_3 - d_1 d_4) + d_1 d_3 - d_2 d_3$,
\item $f_4 = b_0 (d_1 d_4 - d_2 d_4) - d_2 d_3 + d_2 d_4$,
\item $f_5 = a_6 d_2 + b_6 d_1 + d_1 - d_2$, and
\item $f_6 = a_0 d_3 - b_0 d_4 - d_3 + d_4$.
\end{itemize}
By the extension theorem (see, e.g., \cite[Theorem 4 in Chapter 5\S 6]{CLO}) and Proposition \ref{pr:ecuacion}, this means that for every $(d_1,d_2,d_3,d_4) \in \mathbb C^4$ such that \[d_1 d_2 d_3 d_4 (d_1d_4 - d_2d_3) (d_3 - d_4) (d_1 - d_2)\neq 0,\] 
 the diagonal matrix is a product of three chessboard Toeplitz.

However, at the moment, we do not have any reason to discard $M_5$ as a counterexample (and this could happen with $\Toep(M_5) =4$, or even $\Toep(M_5)=5$).
We can only exclude \emph{certain} factorizations with 3 chessboard Toeplitz matrices $M_5$.
Perhaps, for $n \geq 5$ one should decompose the columns and rows of the Toeplitz matrices into more general classes of mutually orthogonal subsets.

\subsection{Toeplitz decomposition of general matrices} We have showed that 
\begin{center}
$3 \leq \Toep_3 \leq  4$ and $3 \leq \Toep_4 \leq 9$.
\end{center}
In both cases we have stratified the problem by rank. For invertible matrices we have invoked Ye and Lim's result Theorem \ref{YeLim1}. For the remaining cases we have produced explicit decompositions involving products of Toeplitz and invertible matrices. Our methods are ad-hoc and we wonder if they can be systematized. In this line, we propose the following conjectures:

We will denote by $\Toep_n^{Inv}$ the minimal number of Toeplitz factors that are needed in order to factor any invertible $n\times n$ matrix.  From Theorem \ref{YeLim1}, we know $\Toep_n^{Inv}\leq 2 \left\lfloor \frac{n}{2} \right\rfloor + 2$
\begin{conj}
\label{super-conjecture}
There exists a constant $C > 0$ such that every matrix can be expressed as a product of an invertible matrix and $C$ Toeplitz matrices. In particular, 
$$\Toep_n \leq \Toep_n^{Inv} + C. $$
\end{conj}

Of course this is a subproblem of Ye-Lim's own one, since $\Toep^{Inv}_n$ is at least the generic lower bound $\left\lfloor \frac{n}{2} \right\rfloor + 1$. We have proven that for $n=3$ and $n=4$, it suffices to take $C = 3$. 

Some matrices that have appeared when factorizing lower rank matrices were ones that we can dub quasi-Toeplitz. We say that a matrix is \emph{quasi-Toeplitz} if it is the result of replacing, in a Toeplitz matrix, some of the rows by zero rows. They seem particularly easy to decompose into Toeplitz factors.

\begin{conj}
Every $n\times n$ matrix is row equivalent to a quasi-Toeplitz matrix. 
\end{conj}

We believe it would be interesting to devise algorithms that receive as input an $n \times n$ matrix and produce a decomposition as a product of a reasonable (linear in $n$) number of Toeplitz factors. Also, it would be interesting to describe methods that receive as input a square $n \times n$ matrix $A$ and a positive integer $s$, and output whether $A$ admits a decomposition into at most $s$ Toeplitz factors or not. Proposition \ref{pr:ecuacion} goes in this direction and translates this problem into the one of deciding whether a system of polynomial equations admits a solution. Nevertheless, due to the size of the polynomial system, this approach is impractical even for small values of $n$ and $s$.

\subsection{The status of the conjecture.}
We were only able to disprove Conjecture \ref{conj1} in the case that $n=3$ but we do believe that counterexamples could exist in all odd dimensions: by choosing $\left \lfloor \frac{n}{2} \right \rfloor+1$ Toeplitz matrices of size $n \times n$, we have $(2n-1)\cdot (\left \lfloor \frac{n}{2} \right \rfloor+1)$ degrees of freedom in total. However, if a factorization $T_1 \ldots T_k$ exists then we can scale each of $T_1, \ldots, T_{k-1}$ by any $\lambda_1, \ldots \lambda_{k-1}$ multiplicative factor and divide $T_k$ by their product. The result will be a new valid factorization, and thus we loose, at least $\left \lfloor \frac{n}{2} \right \rfloor$ dimensions in our count. For the factorization to give a prefixed matrix $M \in \C^{n \times n}$, we need to impose $n^2$ restrictions. When $n$ is \emph{even}, this will leave $n-1$ truly free variables more than needed. When $n$ is \emph{odd}, we get exactly the amount of variables needed. This indicates that the conjecture might be more likely when $n$ is even. Adding one Toeplitz factor for odd $n$ would yield $2n-2$ additional true degrees of freedom, making it somehow more likely than the current conjecture for even $n$. We propose a reasonable substitution for Conjecture \ref{conj1}:
\begin{conj} Let $n >1$. Then
 \[\Toep_n= \left \lceil \frac{n}{2} \right \rceil +1.\]
\end{conj}

A way in which this could be established passes by providing a more explicit understanding of what matrices have the generic behavior prescribed by Ye and Lim.




\subsection*{Acknowledgements}

 The authors would like to thank Professor J. Moro, J. Gouveia and C. Beltr\'an for helpful discussions.


\end{document}